\documentclass[12pt]{amsart}

\IfFileExists{srcltx.sty}{\usepackage{srcltx}}

\numberwithin{equation}{section}

\usepackage[latin1]{inputenc}
\usepackage{xspace,amssymb,amsfonts,euscript}
\usepackage{amsthm,amsmath}
\usepackage{palatino}
\usepackage{euscript}
\usepackage{MnSymbol}

\input xy \xyoption {all}

\usepackage{tikz}

\usepackage{pgflibraryarrows}

\RequirePackage{color}
\definecolor{myred}{rgb}{0.75,0,0}
\definecolor{mygreen}{rgb}{0,0.5,0}
\definecolor{myblue}{rgb}{0,0,0.65}

\RequirePackage{ifpdf}
\ifpdf
  \IfFileExists{pdfsync.sty}{\RequirePackage{pdfsync}}{}
  \RequirePackage[pdftex,
   colorlinks = true,
   urlcolor = myblue, 
   citecolor = mygreen, 
   linkcolor = myred, 
   pagebackref,
   bookmarksopen=true]{hyperref}
\else
  \RequirePackage[hypertex]{hyperref}
\fi

\RequirePackage{ae, aecompl, aeguill} 

    \def\CM{{\mathbb{C}}}

    \def\FM{{\mathbb{F}}}
  \def\gg{{\mathfrak g}}

    \def\QM{{\mathbb{Q}}}

    \def\ZM{{\mathbb{Z}}}

    \def\NC{{\mathcal{N}}}
    \def\OC{{\mathcal{O}}}

\def\FS{{\EuScript F}}
\def\GS{{\EuScript G}}

\def\XS{{\EuScript X}}

\def\a{\alpha}

\def\G{\Gamma}

\def\l{\lambda}

\newcommand{\nc}{\newcommand} \newcommand{\renc}{\renewcommand}

\newcommand{\rdots}{\mathinner{ \mkern1mu\raise1pt\hbox{.}
    \mkern2mu\raise4pt\hbox{.}
    \mkern2mu\raise7pt\vbox{\kern7pt\hbox{.}}\mkern1mu}}

\def\mini{{\mathrm{min}}}

\def\reg{{\mathrm{reg}}}

\newcommand{\elem}[1]{\stackrel{#1}{\longto}}

\def\ov{\overline}
\def\un{\underline}

\def\to{\rightarrow}

\def\longto{\longrightarrow}

\nc{\triright}{\stackrel{[1]}{\to}}
\nc{\longtriright}{\stackrel{[1]}{\longto}}

\nc{\Br}{\mathcal{B}}
\nc{\HotRR}{{}_R\mathcal{K}_R}
\nc{\HotR}{\mathcal{K}_R}
\nc{\excise}[1]{}
\nc{\defect}{\text{df}}
\nc{\h}[1]{\underline{H}_{#1}}

\nc{\Ga}{\mathbb{G}_a} 
\nc{\Gm}{\mathbb{G}_m} 

\nc{\Perv}{{\mathbf{P}}}

\nc{\IH}{{\mathrm{IH}}}

\nc{\ic}{\mathbf{IC}}

\nc{\gl}{{\mathfrak{gl}}}
\renc{\sl}{{\mathfrak{sl}}}
\renc{\sp}{{\mathfrak{sp}}}

\nc{\HBM}{H^{BM}}

\DeclareMathOperator{\For}{For} 

\newtheorem{theorem}{Theorem}[section]
\newtheorem{lemma}[theorem]{Lemma}
\newtheorem{proposition}[theorem]{Proposition}
\newtheorem{corollary}[theorem]{Corollary}

\theoremstyle{definition}

\newtheorem{definition}[theorem]{Definition}

\newtheorem{conjecture}[theorem]{Conjecture}

\theoremstyle{remark}
\newtheorem{remark}[theorem]{Remark}

\def\uk{\underline{k}}

\def\pt{{\mathrm{pt}}}

\def\H{{\mathsf{H}}}

\def\Mod{\mathrm{mod}^\ZM}

\newcommand{\simto}{\stackrel{\sim}{\longrightarrow}}

\nc{\om}{\varpi^{\vee}}

\begin{document}

\title[Kumar's criterion modulo $p$]{Kumar's criterion modulo $p$}

\author{Daniel Juteau}
\address{LMNO, Université de Caen Basse-Normandie, CNRS, BP 5186,
  14032 Caen Cedex, France}
\email{daniel.juteau@math.unicaen.fr}
\urladdr{http://math.unicaen.fr/~juteau}

\author{Geordie Williamson}
\address{Mathematical Institute, University of Oxford, 24-29 St Giles', Oxford, OX1 3LB, UK}
\email{geordie.williamson@maths.ox.ac.uk}
\urladdr{http://people.math.ox.ac.uk/williamsong}

 \maketitle


\begin{abstract}
We prove that equivariant multiplicities may be used to determine whether attractive fixed points on
$T$-varieties are $p$-smooth. This gives a combinatorial criterion for
the determination of the $p$-smooth locus of Schubert varieties for
all primes $p$.
\end{abstract}
\maketitle

\section{Introduction}
\label{sec:introduction}

Let $X$ be an $n$-dimensional complex algebraic variety equipped with
its classical topology and let $p$ be a prime number. A point $x \in X$ is
\emph{$p$-smooth} if one has an isomorphism
\[
\H^{\bullet}(X, X \setminus \{ x \}; \FM_p) \cong \H^{\bullet}(\CM^n, \CM^n \setminus \{ 0 \}; \FM_p).
\]
The \emph{$p$-smooth locus} is the largest open subset of $X$
consisting of $p$-smooth points. Similarly one defines \emph{$\ZM$-smooth}
and \emph{rationally smooth} by replacing the $\FM_p$-coefficients
above by $\ZM$ and $\QM$ respectively. One has inclusions
\cite[Section 8.1]{FW}
\[
\begin{array}{c}\text{smooth}\\ \text{locus}\end{array} 
\subset \begin{array}{c}\text{$\ZM$-smooth}\\ \text{locus}\end{array}
\subset \begin{array}{c}\text{$p$-smooth}\\ \text{locus}\end{array}
\subset \begin{array}{c}\text{rationally}\\ \text{ smooth locus}\end{array}
\]
all of which are strict in general. (These inclusions are already strict for Kleinian surface singularities: all such singularities are rationally smooth (being finite quotient singularities), however a singularity of type $A_n$ is $p$-smooth if and only if $p \nmid n+1$ and a singularity of type $E_8$ is $\ZM$-smooth.)

The variety $X$ is $p$-smooth if and only if the constant sheaf on $X$
with coefficients in $\FM_p$ is Verdier self-dual (up to a shift). It
follows that if $X$ is $p$-smooth and compact, then Poincar\'e duality
holds for the cohomology of $X$ with coefficients in $\FM_p$ (and
similarly, if $X$ is $\ZM$- or rationally smooth; in the case of
$\ZM$, it is a derived duality).

As is clear from the definition, the notion of $p$-smoothness is
topological in nature. In general it seems difficult to decide whether
a point $x \in X$ is $p$-smooth. In this paper we give a
combinatorial criterion which often enables one to decide whether an isolated
fixed point of a $T$-variety is $p$-smooth.



Let $T$ denote an algebraic torus, let $S$ denote the symmetric
algebra on the character lattice of $T$, and let $Q$ denote its fraction
field. We view $S$ as a graded algebra with the characters of $T$ in
degree $2$. We assume that $X$ is a normal $T$-variety and that $T$
has finitely many fixed points on $X$.

To any fixed point $x \in X^T$ one may associate its ``equivariant
multiplicity'' $e_xX \in Q$ which is obtained by localising the
fundamental class in equivariant Borel-Moore homology (see Definition
\ref{def:eqmult}). It is of the form
\begin{equation}
\label{eq:eqmult intro}
e_xX = \frac{f_x}{\chi_1 \chi_2 \dots \chi_m}
\end{equation}
with $f_x \in S$, where $\chi_1, \chi_2, \dots, \chi_m$ are characters of $T$ which
occur in the tangent space of $X$ at $x$. This rational function is
homogeneous of degree $-2n$ (recall that we have doubled
degrees, and that $n = \dim_\CM X$).

For example, if $x \in X$ is a smooth point, then 
\[
e_xX = \frac 1 {\det T_x X} = \frac{1}{\chi_1 \chi_2 \dots \chi_n}
\]
where $\chi_1, \chi_2, \dots, \chi_n$ are all the characters of $T$ on the
tangent space of $X$ at $x$.

Even for singular points $x \in X$ the
equivariant multiplicity is often readily computed; indeed, if $\pi: Y
\to X$ is a proper surjective $T$-equivariant morphism of finite
degree $d$ and $Y^T$ is finite then, for $x \in X^T$, we have
\cite[Lemma 16]{Brion}
\begin{equation}
\label{degree d}
e_xX = \frac 1 d \sum_{\substack{y\in Y^T\\ \pi(y) = x}} e_yY.
\end{equation}

It is particularly interesting to consider the case of
Schubert varieties in flag varieties for reductive algebraic or
Kac-Moody groups $G$ with the action of a maximal torus $T$.
In this case, $T$-fixed points on the flag
variety are in bijection with elements of the Weyl group and 
(thanks to the existence of Bott-Samelson resolutions) there exist
purely algebraic formulas for the equivariant multiplicity
in terms of the Weyl group action on the root system.

For Schubert varieties, the rational functions
$e_x X$ were first defined in 1987 by Kumar, using the nil-Hecke
ring \cite{Kumar}. He showed that they may
be used to detect the smooth and rationally smooth loci of Schubert
varieties. More precisely, if the fraction in \eqref{eq:eqmult intro}
is reduced, then
\[
\text{$x \in X$ is smooth} \Longleftrightarrow f_x = 1.
\]
Moreover, if $U := X \setminus \{ x \}$ is rationally smooth then:
\[
\text{$x \in X$ is rationally smooth} \Longleftrightarrow f_x \text{
  is a constant.}
\]

This result was generalised to cover $T$-varieties with
isolated fixed points and finitely many one-dimensional orbits by
Arabia \cite{Arabia}. A more general statement due to Brion
\cite{Brion} gives a relative version in terms of fixed points
under codimension one subtori.

In the rationally smooth case, it is natural
ask what other geometric or topological significance this numerator
might have. In the case of a rationally smooth point in a Schubert
variety in a flag variety of a simple algebraic group not of type
$G_2$, Kumar interpreted the numerator of the equivariant multiplicity
as the multiplicity of the point \cite[Remark 5.3]{Kumar}.

The goal of this paper is to show that equivariant multiplicities
may also be used to detect $p$-smoothness:

\begin{theorem} 
\label{thm:main}
Let $X$ be an affine $T$-variety with an attractive (hence unique) fixed
point $x$. If $U = X \setminus \{x\}$ is $p$-smoooth and if
$H_T^{\bullet}(U; \ZM)$ is free of $p$-torsion, then
\[
 \text{$x \in X$ is $p$-smooth}
\Longleftrightarrow
f_x \in \ZM \text{ and }p \nmid f_x, 
\]
where $f_x$ is the numerator of the fraction \eqref{eq:eqmult intro},
assumed to be reduced.
\end{theorem}

In the theorem, $H_T^{\bullet}(U; \ZM)$ denotes the $T$-equivariant
cohomology of $U$ with coefficients in $\ZM$ (see Section
\ref{sec:eqcoh}).  Note that by a theorem of Sumihiro \cite{Sumihiro},
any attractive $T$-fixed point on a normal $T$-variety has a
$T$-stable affine open neighbourhood, so the requirement that $X$ be
affine is mostly harmless.

On the other hand, requiring that $H_T^{\bullet}(U; \ZM)$ be free of $p$-torsion may seem
quite restrictive (and of course does not appear in the
criterion for rational smoothness). Fortunately this condition always holds for (normal
slices in) Schubert varieties, as can be proved \cite{FW} using parity sheaves
\cite{JMW2} (see Section \ref{sec:schubert} for a precise
statement). Hence we obtain a combinatorial recursive criterion to
determine the $p$-smooth locus of Schubert varieties, refining
Kumar's orginal criterion.

To be more specific, let $G \supset B \supset T$ be a complex
reductive algebraic group with a Borel subgroup and maximal torus.
The set of $T$-fixed points on the flag variety $G/B$ may be
identified with the Weyl group $W$, and we have the Bruhat decomposition
$G/B = \bigsqcup_{x\in W} BxB/B$, whose closure relation is given by the
Bruhat order $\leq$ on $W$. Given a Schubert variety
$X_z := \ov{BzB}/B$, let us denote by $f_{y,z}$ the numerator appearing
in the equivariant multiplicity of $X_z$ at the $T$-fixed point $y \leq z$.

\begin{theorem}
\label{thm:main schubert}
With the above notation, a $T$-fixed point $x$ in $X_z$ is $p$-smooth
if and only if for all $y$ in the interval $[x,z]$, the numerator $f_{y,z}$ is an
integer and is not divisible by $p$.
\end{theorem}

Taking into account Kumar's criterion for smoothness, we obtain:

\begin{corollary}
\label{cor:Z}
The smooth and $\ZM$-smooth loci of Schubert varieties coincide.
\end{corollary}

We remark that both Theorem \ref{thm:main schubert} and Corollary
\ref{cor:Z} also hold for Schubert varieties in Kac-Moody flag varieties.

Dyer gives in \cite{Dyer} a detailed combinatorial analysis of
equivariant multiplicities for Kac-Moody flag varieties, and derives a 
criterion, in terms of the Bruhat graph,
for a point to be rationally smooth. He also gives
explicit formulas (in terms of ``generalised binomial coefficients'')
for the numerators in this case. Combining Dyer's
results with our main theorem yields:

\begin{theorem} 
\label{thm:pQ}
Let $X$ be a Schubert variety in a  (finite) flag variety $G/B$ for a
semi-simple algebraic group $G$. Then its $p$-smooth locus is the same
as its rationally smooth locus for the following primes $p$:
\begin{enumerate}
\item all $p$ if $G$ contains only components of types $A$, $D$ and $E$,
\item $p \ne 2$ if $G$ does not contain a component of type $G_2$,
\item $p \ne 2, 3$ in general.
\end{enumerate}
\end{theorem}

For a fixed Kac-Moody Schubert variety, one may use our main theorem to
compare the rationally and $p$-smooth loci for any $p$ using Dyer's formulas for
the numerators in equivariant multiplicities. However in the infinite
family of Schubert varieties occuring in a non-finite Kac-Moody flag
there is no uniform bound for the primes dividing the numerators.
For example, in the affine flag variety of $SL_2(\CM)$ all natural numbers
appear.

Note that this theorem has been obtained independently in \cite{FW} using moment
graph techniques. It answers a question of Dyer about a geometric
interpretation of the numerator in the equivariant
multiplicity. It also answers in the affirmative a question
of Soergel: in \cite{Soergel} he asks whether a rationally smooth
Schubert variety in $G/B$ is $p$-smooth, as long as $p$ is larger than
the Coxeter number of $G$.

It is a result due to Peterson (which is reproved by Dyer in
\cite{Dyer}) that the smooth and rationally smooth loci of Schubert
varieties in simply laced type coincide. Combining Corollary \ref{cor:Z}
with \cite{FW}'s proof of Theorem \ref{thm:pQ}, we obtain
another independent proof of this result.

Lastly, let us also point out that the above result gives a quick
proof of a conjecture of Malkin, Ostrik and Vybornov about non
smoothly equivalent singularities in the affine Grassmannian
\cite{MOV} (see Section \ref{sec:schubert}). The point is that, for a
rationally smooth $T$-fixed point $x$ in a Schubert variety, the set of
primes dividing $f_x$ (and hopefully $f_x$ itself, see Conjecture
\ref{conj:torsion} below) is an invariant of the singularity up to smooth
equivalence.

The proof of Theorem \ref{thm:main} is quite straightforward,
but makes heavy use of the equivariant sheaves and
localisation. A key ingredient is that the universal coefficient
theorem (relating cohomology over $\QM_p$, $\ZM_p$ and
$\FM_p$) allows one to view certain equivariant cohomology groups with
coefficients in $\ZM_p$ as lattices inside the cohomology over
$\QM_p$. We then use valuation arguments to deduce the
theorem. However, as mentioned before, the fact that we can always
apply our main theorem to Schubert varieties relies on the results of
\cite{FW}.

Finally, let us note that in \cite{JW} several examples are
computed. (It was on the basis of these examples that we were led to
conjecture Theorem \ref{thm:main}). In these examples we have observed
an even stronger connection which we would like to state as a conjecture. Let $N$
denote a $T$-invariant affine normal slice to a Bruhat cell $BxB/B$ in the
Schubert variety $X_z$ and let $f_{x,z}$ be as above.

\begin{conjecture}
\label{conj:torsion}
If $U := N \setminus \{x\}$ is smooth and $N$ is rationally smooth, then the order of the torsion
subgroup of $\H^{\bullet}(U, \ZM)$ is $f_{x,z}$.
\end{conjecture}

A local version of the conjecture is the following: if $U$ is
$p$-smooth, then the order of the torsion subgroup of
$\H^{\bullet}(U, \ZM_p)$ is equal to the $p$-part of $d_x$.
Also, one could hope that this conjecture
is true for an attractive fixed point $x$ in an arbitrary
normal affine variety $X$, assuming that $U := X \setminus \{x\}$
is smooth (resp. $p$-smooth) and that $\H^\bullet_T(U,\ZM)$ is free (resp. free of
$p$-torsion).

\subsection{Acknowledgements} We would like to thank Michel Brion who
pointed out errors in previous attempts to prove the main theorem (dating back to
2008!). We also benefited from conversations with Matthew Dyer, Sam
Evens and Peter Fiebig. We are grateful to Shrawan Kumar for feedback
on a preliminary version of this article.

\section{Notation}

All varieties will be complex algebraic equipped with
their classical (metric) topology. We denote by $p$ a prime number,
$\FM_p$ the finite field with $p$ elements and $\ZM_p$ (resp. $\QM_p$)
the $p$-adic integers (resp. numbers). Throughout, $k$ denotes a ring
of coefficients (usually $\ZM$, $\FM_p$, $\ZM_p$ or $\QM_p$). If $M$ is a
$\ZM_p$-module we denote by $M_{\QM_p} := M \otimes_{\ZM_p} \QM_p$ its
extension of scalars to $\QM_p$.

\section{Cohomology and equivariant cohomology}
\label{sec:eqcoh}

Given a variety $X$ we denote by $D(X) = D(X;k)$ the derived category of
$k$-sheaves on $X$. Given $\FS \in D(X)$ we denote by
$\H^{\bullet}(X, \FS)$ its (hyper)cohomology, a graded $k$-module.
If $X$ is acted on by a linear algebraic group $G$ we
denote by $D_G(X) = D_G(X;k)$ the equivariant derived category of
$k$-sheaves on $X$ \cite{BL}. If $EG$ denotes a classifying
space for $G$ then $D_G(X;k)$ may be described as the full subcategory
of $D^b(X \times_G EG)$ consisting of those objects $\FS$ such that
$q^*\FS \cong p^*\GS$ for some $\GS \in D(X)$,
where $p$ and $q$ denote the \emph{p}rojection
and \emph{q}uotient morphisms:
\[
\xymatrix@=0.7pt{
 & & & X \times EG \ar[drrr]^(0.5)q \ar[dlll]_(0.6)p \\
X & & & & & & X \times_G EG
}
\]
We denote by $\For = p_*q^*: D_G(X) \to D(X)$ the forgetful
functor. Given any $\FS \in D_G(X)$ we may consider its equivariant
cohomology
\[
\H^{\bullet}_G(X, \FS) := \H^\bullet(X\times_G EG, \FS) \in \H^{\bullet}_G(\pt)-\Mod
\]
where $\H^{\bullet}_G(\pt)-\Mod$ denotes the category of graded modules
over the graded $k$-algebra $\H^{\bullet}_G(\pt)$. Throughout we write
$\H^{\bullet}(X, \FS)$ instead of $\H^{\bullet}(X, \For(\FS))$.

Now suppose that $G = T$ is an algebraic torus. Let $\XS = \XS^*(T)$ denote
its character lattice (a free $\ZM$-module), $\XS_k$ the extension of scalars of $\XS$ to $k$ and $S = S_k =
S(\XS_k)$ the symmetric algebra of $\XS_k$. If we view $S$ as a graded
algebra with $\deg \XS_k = 2$ then we have a canonical isomorphism of
graded rings
\[
\H^{\bullet}_T(\pt) = S.
\]

As is well known, if $\H^{\bullet}_G(X, \FS)$ is a free
$\H^{\bullet}_G(\pt)$-module, then the ordinary cohomology of $X$ with
coefficients in $\For(\FS)$ is obtained from the equivariant
cohomology by extension of scalars:
\[
\H^{\bullet}(X, \FS) \cong \H^{\bullet}_G(X, \FS)
\otimes_{\H^{\bullet}_G(\pt)} k.
\]
We will need a mild extension of this result when $G = T$ which we
were unable to find in the literature:

\begin{proposition}
\label{prop:equiextend}
Fix a basis $e_1, \dots, e_n$ of $\XS =
  \XS(T)$. Let $\FS \in D_T(X)$ and suppose that the images of $e_1, \dots,
  e_n$ in $S$ give a regular sequence for $\H_T^{\bullet}(X,
  \FS) \in S-\Mod$. Then we have an isomorphism
\[
\H^{\bullet}(X, \FS) \cong \H^{\bullet}_T(X, \FS)
\otimes_S k = \H^{\bullet}_T(X, \FS) / (S^+\H^{\bullet}_T(X, \FS))
\]
where $S^+ = \bigoplus_{i > 0} S^i$ denotes the augmentation ideal
generated by homogeneous elements of strictly positive degree.
\end{proposition}

\begin{proof} First suppose that $\pi : E \to B$ is a (topological)
  $\CM^*$-fibration and that $\FS \in D(B)$. We have a spectral
  sequence
\begin{equation} \label{chernclassss}
\H^p(\CM^*) \otimes \H^q(B, \FS) \Rightarrow \H^{p+q}(E, \pi^* \FS)
\end{equation}
with differential induced by multiplication by the Chern class
$c_1(\pi)$.

Now suppose that $T = \CM^*$ so that $\XS(T) = \ZM e$.
If we apply this to the
$\CM^*$-fibration $q: X \times E\CM^* \to
X \times_{\CM^*} E\CM^*$ then, by considering the pull-back diagram
\[
\xymatrix{
X \times E\CM^* \ar[r] \ar[d]^q & E\CM^* \ar[d] \\
X \times_{\CM^*}E\CM^* \ar[r] & B\CM^*}
\]
we see that the first Chern class of $q$ acts on $\H^{\bullet}_T(X,
\FS)$ as multiplication by the image of $e$ in $S$. 
Multiplication by $e$ is injective on $\H^{\bullet}_T(X,
\FS)$ because $(e)$ is a regular sequence and it follows that
\[
\H^{\bullet}(X, \FS) = H^\bullet(X \times E \CM^*, q^*\FS) = 
\H^{\bullet}_{\CM^*}(X, \FS)/ (e \cdot \H^{\bullet}_{\CM^*}(X, \FS)).
\]

We now turn to the general case. Let $T = \CM^* \times \dots \times
\CM^*$ be the splitting of $T$ corresponding to the basis $(e_1,
\dots, e_n)$ of $\XS$ and let $T_i$ denote the subtorus
consisting of the last $i$ copies of $\CM^*$. In other words $T_i = \ker e_1 \cap \dots \cap \ker e_{n-i}$. Let $q_i$ denote the
quotient maps
\[
X \times ET \stackrel{q_0}{\to} X \times_{T_1} ET \stackrel{q_1}{\to} X
\times_{T_2} ET \stackrel{q_2}{\to} \dots \stackrel{q_{n-2}}{\to} X \times_{T_{n-1}} ET \stackrel{q_{n-1}}{\to} X \times_T ET
\]
and set
\[
H_i := H^\bullet(X \times_{T_i} ET, q_i^*q_{i+1}^* \dots q^*_{n-1}
\FS) \text{ and } H_n := H^\bullet_T(X, \FS).
\]
By induction, our regular sequence assumption and the spectral
sequence \eqref{chernclassss} we have
\begin{gather*}
H_i =  H_{i+1}/e_iH_{i+1}= H_n / (e_1,\dots, e_i)H_n.
\end{gather*}
Hence
\[
H^\bullet(X, \FS) = H^\bullet(X \times ET, q^*\FS) = H_0 =
H_n / (e_1,\dots, e_n)H_n
\]
as claimed.
\end{proof}


\section{Equivariant multiplicities}
\label{sec:eqmult}

In this section $T$ denotes a complex torus and $X$ is an
irreducible $n$-dimensional $T$-variety. In this section we
always take equivariant cohomology with coefficients in $k = \ZM$. Given a
$T$-variety $Y$, its equivariant constant and dualising sheaves are
denoted $\uk_Y$ and $\omega_Y$ respectively. We have
$\H_T^{-m}(Y; \omega_Y) = \H^T_{m}(Y)$
where $\H^T_{\bullet}(Y)$ denotes equivariant Borel-Moore
homology. Although we never make use of this isomorphism,
it may provide an intuitive aid for the reader below. 

We first recall the definition of the equivariant canonical class of $X$. 
Let $X^\reg$ denote the smooth locus of $X$, then $X^\reg$ has a
canonical orientation, and hence a canonical class $\mu_{X^\reg} \in
\H^{-2n}_T(X^\reg,\omega_{X^\reg})$. It is straightforward to see
that the restriction map
\[
r: \H^{-2n}_T(X,\omega_{X}) \to \H^{-2n}_T(X^\reg,\omega_{X^\reg})
\]
is an isomorphism.

\begin{definition} The \emph{equivariant canonical class} $\mu_X \in
  \H^{-2n}_T(X,\omega_{X})$ is defined to be the inverse image of
  $\mu_{X^\reg}$ under the isomorphism $r$.\end{definition}

For the rest of this section we assume that $X^T$ is finite. 

Set $U = X \setminus X^T$ and let $i$ (resp. $j$) denote the inclusion
of $X^T$ (resp. $U$). Given any $\FS \in D^t_T(X)$ we have a standard
triangle
\[
i_*i^! \FS \longto \FS \longto j_* j^* \FS \rightsquigarrow
\]
If we take $\FS = \omega_X$  the above triangle may be rewritten as
\[
i_*\uk_{X^T} \to \omega_X \to j_* \omega_U \rightsquigarrow
\]
because $i^!$ and $j^* = j^!$ preserve the dualising sheaf. Taking
equivariant (hyper)cohomology we obtain a long exact sequence
\[
\cdots \longto  H_T^m(X^T) \longto \H^m_T(X,\omega_X) \longto
\H_T^m(U,\omega_U) \longto H_T^{m+1}(X^T) \longto \cdots
\]
Standard arguments (see e.g. \cite{Brion, FW}) show that
$\H_T^\bullet(U,\omega_U)$ is a torsion module over $S$. Since
$\H_T^\bullet(X^T)$ is a free $S$-module (of rank $|X^T|$), the above
long exact sequence is in fact a short exact sequence of $S$-modules:
\[
0 \longto H_T^m(X^T) \longto \H^m_T(X,\omega_X) \longto
\H_T^m(U,\omega_U) \longto 0
\]
and if we tensor with $Q$, the fraction field of $S$, we obtain an isomorphism
\[
i_*: \bigoplus_{x \in X^T} Q = H_T^\bullet(X^T) \otimes_S Q \simto
\H^\bullet_T(X,\omega_X) \otimes_S Q.
\]
Hence we can find rational functions $(e_xX)_{x \in X^T} \in
\bigoplus_{x \in X^T} Q$ such that $i_*( (e_xX)_{x \in X^T}) = \mu_X
\otimes 1$.

\begin{definition}
\label{def:eqmult}
The \emph{equivariant multiplicity} of $x \in X$ is the
rational function $e_xX \in Q$.
\end{definition}

For further discussion about properties of the equivariant
multiplicity see the papers \cite{Arabia} and \cite{Brion}. Note that
in \cite{Brion}, Brion works instead with equivariant Chow groups,
however this may be seen to be equivalent to the above construction
using the cycle map from the equivariant Chow group to equivariant
Borel-Moore homology \cite[Section 2.8]{EG}.

\section{Filtrations and valuations} \label{sec:val}

Let $M$ be a free $\ZM_p$-module. Then $M$ is a lattice in the
${\QM_p}$-vector space $M_{\QM_p} := {\QM_p} \otimes_{\ZM_p} M$, and we have a
filtration:
\[
\dots \supset p^{-1} M \supset M \supset pM \supset \dots
\]

\begin{definition}
The \emph{valuation $v(m) = v_M(m)$ of $m \in M_{\QM_p}$ relative to $M$}
is the greatest $k \in \ZM$ such that $m\in p^k M$, or $+\infty$ if
$m = 0$.
\end{definition}

For example, we have $m\in
M$ if and only if $v(m) \geq 0$. For $d\in{\QM_p}$ and $m \in M_{\QM_p}$, we
have
\[
v_M(dm) = v_p(d) + v_M(m)
\]
and for $m_1$, $m_2 \in M$ we have
\[
v_M(m_1 + m_2) \geq \min(v_M(m_1), v_M(m_2)).
\]

Now suppose that $S = S(\XS)$ is the symmetric algebra over ${\ZM_p}$ on a
free $\ZM$-module $\XS$. Given $f, g \in S_{{\QM_p}}$ it is straightforward to
check that
\[
v_S(fg) = v_S(f) + v_S(g).
\]
It follows that, if $M$ is a free $S$-module, then
\[
v_M(fm) = v_S(f) + v_M(m)
\]
for $m \in M_{{\QM_p}}$ and $f \in S_{{\QM_p}}$.

\section{Proof of Theorem \ref{thm:main}}

Recall that we assume that $X$ is irreducible, affine and
$n$-dimensional and that $x \in X$ is an
attractive (hence unique) $T$-fixed point, and that $U = X \setminus
\{x\}$.
By Kumar's criterion and the fact that $p$-smoothness implies rational
smoothness, Theorem \ref{thm:main} is equivalent to the following:

\begin{theorem} \label{thm:main'}
Assume that $U$ is $p$-smoooth, that
$H_T^{\bullet}(U; \ZM)$ is free of $p$-torsion and that $X$ is
rationally smooth, so that the numerator $d$ in the equivariant
multiplicity is an integer. Then
\[
p \nupdownline d \Longleftrightarrow \text{$x \in X$ is $p$-smooth.}
\]
\end{theorem}

Throughout this section, we always take coefficients in $\ZM_p$
unless otherwise stated. In particular, $\omega_X$ (resp. $\omega_U$)
denotes the $T$-equivariant dualising complex with coefficients in
${\ZM_p}$.
In Section \ref{sec:eqmult} we saw the short exact sequence coming from the standard
distinguished triangle for the decomposition $X = U \cup \{ x \}$:
\begin{equation} \label{eq:ses}
0 \longto S = \H^{\bullet}_T(\pt)\elem{\varphi} \H_T^\bullet(X,\omega_X) \elem{r}
\H_T^\bullet(U,\omega_U) \longto 0
\end{equation}
We assume from now on that $X$ is rationally smooth and abbreviate
\[
\H := \H_T^\bullet(X,\omega_X).
\]
By Kumar's criterion we can write
\begin{equation} \label{eq:kumarfrac}
e_xX = \frac{d}{\pi}
\end{equation}
where $d \in \ZM$, $\pi$ is a product of $n$ characters, and the
fraction is assumed to be reduced.

\begin{remark} \label{rem:pipositive}
In other words, we simplify
the fraction until at most one of $d$ and $\pi$ has positive
valuation. We will see in Lemma \ref{lem:valpi} that, in fact, under the assumptions of
the theorem only $d$ can have positive valuation.
\end{remark}

\begin{remark}
If there is a finite number of one-dimensional orbits,
then there are exactly $n$ of them and $\pi$ is the product of the
corresponding characters up to some scalar multiple (which may be
needed to simplify the fraction). However, our proof also applies when
there is an infinite number of one-dimensional orbits.
\end{remark}

\begin{lemma} \label{lem:ratsmooth}
As $X$ is rationally smooth, we have $\H_{{\QM_p}} \cong S_{{\QM_p}}[2n]$.  
\end{lemma}

\begin{proof}
  By definition, $X$ is rationally smooth if and only if $ \omega_{X,\QM_p} \cong
  {\un \QM_{p,X}}[2n]$. If this is the case then
\[
\H_{\QM_p} \cong \H_T^{\bullet}(X, {\QM_p})[2n] \cong \H_T^{\bullet}(\{x\},
{\QM_p})[2n] \cong S_{{\QM_p}}[2n]
\]
where the first isomorphism follows from the universal coefficient
theorem, and the second follows because $x \in X$ is an attractive
fixed point (and so $X$ retracts equivariantly onto $x$).
\end{proof}

\begin{lemma} In \eqref{eq:ses} all modules are free over ${\ZM_p}$.
\label{lem:ofree}
\end{lemma}

\begin{proof}
  Certainly $S$ is ${\ZM_p}$-free and $\H^{\bullet}_T(U, \omega_U) \cong
  \H^{\bullet}_T(U, {\ZM_p})$ is ${\ZM_p}$-free by assumption. Hence $\H$ is
  ${\ZM_p}$-free, being an extension of  $S$ and $\H^{\bullet}_T(U,
  \omega_U)$.
\end{proof}

\begin{lemma}
\label{lem:pi}
The $S$-module $\H^{\bullet}_T(U, \omega_U)$ is
  annihilated by $\pi$ and in $\H$ we have the relation
\[d \cdot \varphi(1) =  \pi \cdot \mu_X.\]
\end{lemma}

\begin{proof}
By the universal coefficient theorem, if we tensor \eqref{eq:ses} over
${\ZM_p}$ with ${\QM_p}$ we obtain the corresponding short exact sequence with
coefficients in ${\QM_p}$:
\begin{equation} \label{eq:sesK}
0 \longto S_{{\QM_p}} \elem{\varphi} \H_T^\bullet(X,\omega_{X,\QM_p}) \elem{r}
\H_T^\bullet(U,\omega_{U,\QM_p}) \longto 0
\end{equation}
However, $X$ is rationally smooth and hence
$\H_T^\bullet(X,\omega_{X,\QM_p}) \cong S_{{\QM_p}} [2n]$ by Lemma
\ref{lem:ratsmooth}. By the definition of the
equivariant multiplicity,  
\eqref{eq:sesK} this short exact sequence has the form:
\begin{equation} \label{eq:ses2}
0 \longto S_{{\QM_p}} \elem{\varphi} S_{{\QM_p}} [2n] \elem{r}
S_{{\QM_p}}/(\pi) [2n] \longto 0
\end{equation}
where $\varphi(1) = {\pi}/{d}$.

Now, by Lemma \ref{lem:ofree} all modules in \eqref{eq:ses} are free
over ${\ZM_p}$. Hence we have a commutative diagram with vertical injections
\[
\xymatrix{
0 \ar[r] & S  \ar[r]^(.3){\varphi} \ar@{^{(}->}[d] &
\H_T^\bullet(X,\omega_X) \ar[r]^{r} \ar@{^{(}->}[d] 
& \H_T^\bullet(U,\omega_U) \ar[r] \ar@{^{(}->}[d] & 0 \\
0 \ar[r] & S_{{\QM_p}} \ar[r] & S_{{\QM_p}}[2n] \ar[r] & 
S_{{\QM_p}}/(\pi)[2n] \ar[r] & 0
}
\]

We conclude that, in $\H = \H_T^\bullet(X,\omega_X)$ we have the equation
\begin{equation} \label{eq:fund}
d \cdot \varphi(1) =  \pi \cdot \mu_X.
\end{equation}
because this equation holds after extension of scalars to
$\QM_p$, and the above diagram shows that $\H$
injects into the extension of scalars.
\end{proof}

Lemmas \ref{lem:ratsmooth} and \ref{lem:ofree} show that, if we set $M
= S\mu_X$ then $M$ is free as an $S$-module, and gives a lattice
inside $\H_{{\QM_p}}$. We use the lattice $M$ and apply the terminology of
Section \ref{sec:val}.

\begin{lemma}
\label{lem:valbound}
For all $h \in \H$ we have 
\[
v_M(h) \ge -v_p(d).
\]
\end{lemma}

\begin{proof}
Given $h \in \H$ then $\pi \cdot h$ is in the kernel of $r$ by
Lemma \ref{lem:pi} and hence we can write $\pi \cdot h = f \cdot \varphi(1) =
\frac{1}{d}f\pi\cdot \mu_X$ for some $f \in S$. Applying $v_M$ yields
\[
v_S(\pi) + v_M(h) = v_S(f) + v_S(\pi) + v_M(\mu_X) - v_p(d)
\]
and hence
\[
v_M(h) = v_S(f) - v_p(d)
\]
because $v_M(\mu_X) = 0$. The claim now follows because $v_S(f) \ge 0$.
\end{proof}

In the following lemma, we keep the promise made in Remark \ref{rem:pipositive}.

\begin{lemma}
\label{lem:valpi}
We have $v_S(\pi) = 0$.
\end{lemma}

\begin{proof} Because we have assumed that the fraction
  \eqref{eq:kumarfrac} is reduced, if $v_p(d) > 0$ then $v_p(\pi) =
  0$. So we may assume that $v_p(d) = 0$. 
We can write $\pi = p^{v_S(\pi)}{\tilde \pi}$ with
$v_S(\tilde \pi) = 0$. We have
\[
\pi \cdot \mu_X = p^{v_S(\pi)}{\tilde \pi} \cdot \mu_X
= d \cdot \varphi(1).
\]
As $\pi$ annihilates $\H_T^\bullet(U,\omega_U)$ we have 
$p^{v_S(\pi)}r({\tilde \pi} \cdot \mu_X) =0$. But by assumption $\H_T^\bullet(U,\omega_U)$
is torsion-free over $\ZM_p$ and so $r({\tilde \pi} \cdot \mu_X) =
0$. Hence ${\tilde \pi} \cdot \mu_X$ is in the image of $\varphi$ and
so $v_M(\tilde \pi \cdot \mu_X) \ge v_M(\varphi(1))$. Using that
$v_S(\tilde{\pi})= v_p(d) = 0$ and Lemma \ref{lem:pi} it
follows that
\[
0 = v_M(\tilde \pi \cdot \mu_X) \geq v_M(\varphi(1))  = v_M(d \cdot \varphi(1)) 
= v_M(\pi \cdot \mu_X) = v_S(\pi) \ge 0 \]
and so $v_S(\pi) = 0$ as claimed.
\end{proof}

\begin{remark} The relation between the modules $M \subset \H \subset
  \H_{{\QM_p}}$ and the induced filtration by valuation is illustrated in
  Figure \ref{fig:inclusion}. \end{remark}

\begin{figure} \label{fig:inclusion}
\begin{tikzpicture}
\def\mumin{-7}

\begin{scope}
\clip (-5.5,2.5) rectangle (2.5,\mumin - 0.5);


\filldraw[fill=gray!28] (-6,\mumin) rectangle (4,3);

\filldraw[fill=gray!35,draw=black] (-4,0) -- (-4,-2) -- (-3,-2) -- (-3,-3) --
(-2,-3) -- (-1,-3) -- (-1,-5) -- (0,-5) -- (0, \mumin) -- (4, \mumin)
-- (4,3) -- (-4,3) -- (-4,0);
\filldraw[fill=gray!45,draw=black] (0,\mumin) rectangle (4,3);

\node (o) at (0.5,0.5) {$\pi\mu_X$};
\node (phi) at (-3.5,0.5) {$\varphi(1)$};
\node (phi) at (-3.5,-1.5) {$*$};
\node (mu) at (0.5,\mumin+0.5) {$\mu_X$};


\foreach \x in {-5,-4,...,2}
  \draw[very thin, gray] (\x,\mumin-1) -- (\x, 5);

\foreach \y in {1,0,\mumin+1,\mumin}
  \draw[very thin, gray] (-6,\y) -- (4, \y);

\draw (-4,0) -- (-4,-2) -- (-3,-2) --  (-3,-3) --
(-2,-3) -- (-1,-3) -- (-1,-5) -- (0,-5) -- (0, \mumin) -- (4, \mumin)
-- (4,3) -- (-4,3) -- (-4,0);
\draw (0,\mumin) rectangle (4,3);
\draw (-6,\mumin) -- (3,\mumin);

\filldraw[fill=gray!28,draw=gray!28] (-3,-4) circle (0.6cm);
\node at (-3,-4) {$\H_{{\QM_p}}$};

\node at (-1.5,-1) {$\H$};

\filldraw[fill=gray!45,draw=gray!45] (1.5,-3) circle (0.8cm);
\node at (1.5,-3) {$M = S\mu_X$};

\end{scope}

\draw[draw=white] (-5.5,2.5) rectangle (2.5,\mumin - 0.5);

\node at (0.5,2.8) {\small $0$};
\node at (-0.5,2.8) {\small $-1$};
\node at (-1.5,2.8) {\small $-2$};
\node at (-2.3,2.8) {\small $\dots$};
\node at (-3.5,2.8) {\small $-v_p(d)$};

\node at (-1.5,3.4) {filtration by $v_M(m)$};

\node at (-6.2,0.5) {$0$};
\node at (-6.2,\mumin+0.5) {${-2n}$};

\begin{scope}
  \node[rotate=90] at (-6.2,-3) {cohomological degree $\rightarrow$};
\end{scope}

\end{tikzpicture}
\caption{The inclusions $M \subset \H \subset \H_{{\QM_p}}$.}
\end{figure}
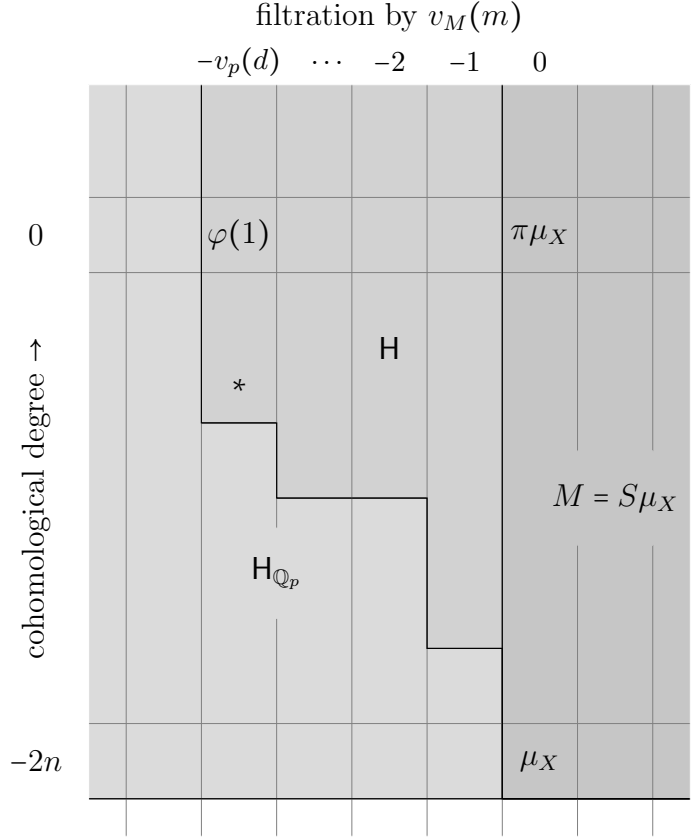

\begin{lemma} \label{lem:reg}
If $(e_1, \dots, e_r)$ denotes a basis for $\XS(T)$ then the images of
$(e_1, \dots, e_r)$ in $S$ give a regular sequence for $\H$.\end{lemma}

\begin{proof}
  Using the inclusion $\H \hookrightarrow \H_{{\QM_p}} = S_{{\QM_p}}$ we see
  that, for all $1 \le i \le r-1$, if $e_{i+1}h \in \langle e_1,
  \dots, e_i \rangle \H$ then $h \in \langle e_1, \dots, e_i \rangle \H$
  which is the condition for $(e_1, \dots, e_r)$ to give a regular sequence.
\end{proof}

\begin{lemma} \label{lem:psmooth}
$X$ is $p$-smooth if and only if $\H^{\bullet}(X,
  \omega_X)$ is torsion-free.\end{lemma}

\begin{proof}
In the proof we denote by $\omega^0_X$ the non-equivariant dualising
complex on $X$ (with $\ZM_p$ coefficients). By definition,
$X$ is $p$-smooth if and only if
\begin{equation}
\label{eq:p}
\forall y \in X, \quad \omega^0_{X,y} \cong \ZM_p[2n].
\end{equation}

By assumption, $U$ is $p$-smooth so this holds for all $y \neq x$ in
$X$. By a standard argument (see for example the attractive
Proposition 2.2 of \cite{FW}) we know that
\[
\H^{\bullet}(\omega^0_{X,x}) \cong \H^{\bullet}(X, \omega^0_X).
\]
Because $X$ is assumed to be rationally smooth, we know that the free
part of $\H^{\bullet}(\omega^0_{X,x})$ is concentrated in degree $-2n$
where it is of rank one. Hence we have \eqref{eq:p} if and only if
$\H^{\bullet}(X, \omega^0_X)$ is torsion-free. \end{proof}

\begin{proof}[Proof of Theorem \ref{thm:main'}]
By Proposition \ref{prop:equiextend} and Lemma \ref{lem:reg} we have
\[
\H^{\bullet}(X, \omega_X) = \H / (S^+\H).
\]
By Lemma \ref{lem:psmooth} $X$ is $p$-smooth if and only if $\H/(S^+\H)$
is torsion-free.

Now choose $m \in \H^i$ and let $\overline{m}$ denote its class in $\H/
(S^+\H)$. Then $v_M(dm) = v_p(d) + v_M(m) \ge 0$ by Lemma
\ref{lem:valbound}.
Hence
$dm \in S\mu_X$. In other words, $d\overline{m} = 0$ unless $i =
-2n$. It follows that multiplication by $d$ annihilates the
torsion in $\H/ (S^+\H)$. If $p \nupdownline d$ then multiplication by
$d$ is an automorphism of $\H/ (S^+\H)$, and hence $\H/ (S^+\H)$ is
torsion-free.

Now let us assume that $p \updownline d$, i.e. $v_p(d) > 0$.
Let $f \in S$ be a homogeneous element of maximal degree such that
$\varphi(1) = f h$ for some $h \in \H$. Note that $h \notin
S^+ \H$.\footnote{A clue as to the position of $h$ in our diagram is
  given by an asterix.}
By Lemma \ref{lem:valbound} and Lemma \ref{lem:pi}, we have
\[
- v_p(d) = v_M(\varphi(1)) = v_S(f) + v_M(h) \geqslant 0 - v_p(d)
\]
hence we have both equalities $v_S(f) = 0$ and $v_M(h) = -v_p(d)$.
In particular, $h \notin \H^{-2n} = \ZM_p \mu_X$. Now,
$v_M(d \cdot h) = 0$ so $d h \in M = S \mu_X$. By the
previous observation, we actually have
$d h \in S^+ \mu_X \subset S^+ \H$.
So the image $\ov h$ of $h$ in $\H / S^+ \H$ is non-zero and
torsion. Since $\H / S^+ \H$ is not torsion-free, $X$ is not $p$-smooth.
\end{proof}

\section{The case of Schubert varieties}
\label{sec:schubert}

In this section, $G$ denotes a connected reductive complex algebraic group, and we make
a choice $G \supset B \supset T$ of a Borel subgroup and a maximal
torus. Let $X = G/B$ be the flag variety and $W$ the Weyl group. For
$w\in W$, let $C_w := BwB/B$ be the corresponding Bruhat cell (which
is an affine space of dimension equal to the length of $w$), with
closure the Schubert variety $X_w = \ov C_w$. We have the Bruhat
decomposition
\[
X = \bigsqcup_{w \in W} C_w.
\]
(More generally, we could take $X$ be a partial flag variety for a
Kac-Moody group, with appropriate modifications.)

Recall from \cite{KL2} (or \cite{Kumar:book} in the Kac-Moody case)
that, for any elements $x \leq w$ in $W$, we can find an affine
neighbourhood $\widetilde N$ of $C_x$ in $X_w$, a closed subset $N$ in
$\widetilde N$ and an isomorphism
\[
C_x \times N \elem{\sim} \widetilde N \subset X_w.
\]

We will use the following result which is proved in \cite[Corollary 8.9]{FW}.

\begin{proposition}
\label{prop:torsion-free}
Suppose that $N$ is $p$-smooth. Then 
$\H^\bullet_T(N\setminus\{x\}, \ZM_p)$ is torsion-free.
\end{proposition}

Hence we can apply our main theorem \ref{thm:main} in the case of
Schubert varieties. Theorem \ref{thm:main schubert} and Corollary
\ref{cor:Z} follow immediately. Theorem \ref{thm:pQ} follows from
\cite[Corollary 3.5]{Dyer} which shows that the numerator of the
equivariant multiplicity at a rationally smooth point in a Schubert
variety for a finite flag variety is always of the form $2^a3^b$, with
$b$ possibly non-zero only if $G$ contains a component of type $G_2$,
and $a = b = 0$ in simply-laced types.

Lastly, in \cite{MOV}, Malkin, Ostrik and Vybornov study minimal
degenerations in affine Grassmannians up to smooth equivalence. They
find the following possibilities: simple singularities of type $A$ in
the codimension $2$ case, minimal nilpotent singularities (the
singularity of $0$ in the closure of the minimal nilpotent orbit in a
simple Lie algebra), and some presumably new singularities which they
call quasi-minimal of type $ac_n$ (of codimension $2n$, arising in
type $C_n$), and $ag_2$ and $cg_2$ (of codimension $4$, arising in
type $G_2$). They compute their local (rational) intersection
cohomology using the Kazhdan-Lusztig algorithm, and their equivariant
multiplicities. It turns out that $ac_n$ (resp. $ag_2$, $cg_2$) has
the same local rational intersection cohomology as a minimal
singularity of type $a_n$ (resp. $a_2$, $c_2$). Malkin, Ostrik and
Vybornov conjectured that the pairs $(ac_n, a_n)$, $(ag_2, a_2)$,
$(ac_2, ag_2)$ and $(c_2, cg_2)$ are not smoothly equivalent. Among
those, only the last one involves rationally smooth singularities. The
numerator for $g_2$ is $18$, whereas the numerator for $cg_2$ is $27$,
so $cg_2$ is $2$-smooth and $g_2$ is not. Hence these singularities
are not smoothly equivalent. To prove their conjecture in the other
cases, one needs either to do a more involved calculation or to use
the geometric Satake correspondence \cite{JW}.

\section{A zoo of (rationally smooth) points}

We conclude with some examples of rationally smooth attractive fixed
points, their equivariant multiplicities and the cohomology of the
complement. It was based on these and other examples that the authors
were led to believe that something like the main theorem
\ref{thm:main} must hold. The reader can also verify that
the more precise Conjecture \ref{conj:torsion} holds in all of these
cases.

\subsection{Smooth points} Let $x = 0 \in X = \CM^n$ with a torus $T$
acting linearly with characters $\chi_1, \dots, \chi_n$. Then
\begin{gather*}
\dim X = n, \\
e_xX = \frac{1}{\chi_1 \chi_2 \dots \chi_n}.
\end{gather*}
The cohomology $H^\bullet(X \setminus \{ x \}; \ZM)$ is given by:
\[
\begin{array}{cccccccc}
H^0 & H^1 & H^2 & \dots & H^{2n-2} & H^{2n-1} \\
\hline
\ZM & 0 &  0 & \dots & 0 & \ZM
\end{array}
\]

\subsection{Kleinian singularities of type $A$} Let $x = 0 \in X =
\CM^2/\mu_{n+1}$ where $\mu_{n+1}$ denotes the group of $n+1$ roots of
unity acting via $\mu \cdot (x,y) = (\mu x, \mu^{-1}y)$. Then $X$ is a
Kleinian surface singularity of type $A_n$. We can embed $X$ as the locus of $(u,v,w)$ in
$\CM^3$ such that $uv = w^{n+1}$. Then $X$ has an attractive $T =
(\CM^*)^2$ action
given by
\[ (\l_1,\l_2) \cdot (u,v,w) = (\l_1 \l_2^n u, \l_1^{-1}\l_2 v, \l_2 w).\]\
The map $\pi : X \to \CM^2$ induced by the projection $(u,v,w) \mapsto (u,v)$ is finite of degree $n+1$. Applying \eqref{degree d} and the case of a smooth point discussed above yields
\begin{gather*}
e_xX = \frac{n+1}{(e_1 + ne_2) (e_2 - e_1)}
\end{gather*}
where $e_1$ and $e_2$ are the characters of $T$ given by $e_i(\l_1, \l_2) = \l_i$.
The cohomology $H^\bullet(X \setminus \{ x \}; \ZM)$ is given by:
\[
\begin{array}{cccccccc}
H^0 & H^1 & H^2 & H^{3} \\
\hline
\ZM & 0 &  \ZM/(n+1) & \ZM
\end{array}
\]
This calculation follows easily from the discussion in 
\cite[Section 3.4]{JMW}.

\subsection{Minimal nilpotent orbit singularities}
Let $G$ denote a connected complex simple algebraic group, $\gg$ its
Lie algebra, $\NC \subset \gg$ its nilpotent cone and $\OC_{\mini}
\subset \NC$ the minimal nilpotent orbit. The closure
$X = \ov{\OC}_{\mini}$ is singular, with unique singular point $0 \in
\gg$. It turns out that $\ov{\OC}_{\mini}$ is only rationally smooth in
types $C_n$ (including types $C_1 = A_1$ and $C_2 = B_2$) and
$G_2$. In this section we discuss what our main theorem has to say in
these cases.

Let $T \subset G$ denote a maximal torus. Then $X = \ov{\OC}_{\mini}$
is a $\tilde T := T \times \CM^*$ variety, where $T$ acts by conjugation and
$\CM^*$ acts by scaling. We write the characters of $T \times \CM^*$
as $\XS^*(T) \oplus \ZM \delta$ where $\XS^*(T)$ denotes the character
lattice of $T$ and $\delta$ denotes the identity character of
$\CM^*$. Finally, let us denote the set of roots by
$\Phi \subset \XS^*(T)$, the subset of long roots by 
$\Phi_{\textrm{lg}} \subset \Phi$ and the Weyl group by $W$.

Fix a Borel subgroup $B$ of $G$ containing $T$ and let $\Phi^+ \subset
\Phi$ denote the non-trivial characters of $T$ which occur in the Lie algebra of $B$.
One can describe the
singularity $X$ as follows \cite{cohmin}: let $\tilde \alpha$ denote the highest
root with respect to $\Phi^+$, and let $\tilde I$ denote the set of simple
roots orthogonal to $\tilde \alpha$. The stablizer of $\tilde \alpha$
in $W$ is the parabolic subgroup $W_{\tilde I}$, and the subgroup of $G$ stabilizing
the root subspace $\gg_{\tilde \alpha} \subset \gg$ is the parabolic subgroup
$P_{\tilde I} = BW_{\tilde I}B$.
Then $X$ is obtained from the line bundle $Y := G \times^{P_{\tilde I}}
\gg_{\tilde \alpha}$ over $G/P_{\tilde I}$ by contracting the null section, and the contraction
morphism $\pi : Y \to X$ is a resolution of singularities. The variety
$Y$ can be seen as the set of pairs $(x,L)$ where $L$ is a line
contained in $\ov \OC_\mini$ and $x$ is an element of $L$.

One may apply \eqref{degree d} to the resolution $\pi$ to find a formula for the
equivariant multiplicity valid in any type:
\begin{equation}
\label{eq:min}
e_0 \;\ov{\OC}_{\min}
= - \sum_{w \in W/W_{\tilde{I}}}
\frac{1}{w((\delta+\tilde{\a})(\prod_{\a \in \Phi^+ \setminus \Phi^+_{\tilde{I}}} \a))} 
\end{equation}
Indeed, $\tilde T$-fixed points in $\pi^{-1}(0)$ must be in the null
section of $Y$ because they are fixed by $\CM^*$. Now the null section
is $G/P_{\tilde I}$, and the $T$-fixed points are in bijection with
$W/W_{\tilde I}$. One finds the sum of the right-hand-side, the case
$w = 1$ corresponding to the line $\gg_{\tilde \alpha}$. Since $Y$ is
a vector bundle over $G/P_{\tilde I}$, the tangent space to $Y$ at
this point can be decomposed into the tangent space to
$G/P_{\tilde I}$ at its base point, whose $\tilde T$-weights are the
$-\alpha$ for $\alpha \in \Phi^+ \setminus \Phi^+_{\tilde{I}}$, and
the tangent space of the fibre $\gg_{\tilde \alpha}$ whose $\tilde
T$-weight is $\delta + \tilde \alpha$. Not that
$\dim G/P_{\tilde I} = \dim \OC_\mini - 1$ is odd, hence the global
minus sign.

The one-dimensional $\tilde T$-orbits are the long root subspaces. Let
now $\varphi : X \to \bigoplus_{\alpha\in\Phi_{\textrm{lg}}}
\gg_\alpha =: V$ be the composition
of the inclusion $X \to \gg$ followed by the natural projection.
By the proof of Theorem 18 in \cite{Brion} this morphism is finite,
and in the types $C_n$ and $G_2$ where $X$ is rationally smooth, it is
surjective. In this case, if $d$ denotes its degree, then
\begin{equation}
\label{eq:eqmult finite}
e_0 X = \frac d {\prod_{\alpha\in\Phi_{\textrm{lg}}}(\delta + \alpha)}.
\end{equation}

We first discuss the case of type $C_n$. So let $G = \textrm{Sp}(2n)$
be the symplectic group and $\sp_{2n}$ its Lie algebra. 
In this case $X$ is isomorphic to $\CM^{2n}/\{\pm 1\}$ (diagonal action) \cite[Section 3.3]{JMW},
and the morphism $X \to Z$ of the last paragraph is identified with
$\CM^{2n}/\{\pm 1\} \to \CM^{2n}/\{\pm 1\}^{2n} \simeq \CM^{2n}$
which is of degree $2^{2n-1}$. In this case it is more convenient to
apply formula \eqref{eq:eqmult finite}. We have
\begin{gather*}
\dim X = 2n, \\
e_0X = \frac{2^{2n-1}}{\prod_{\a \in \Phi_\textrm{lg}} (\delta + \a)}.
\end{gather*}

On the other hand, the cohomology $H^\bullet(X \setminus \{ 0 \}; \ZM)$ is given by:
\[
\begin{array}{ccccccccc}
H^0 & H^1 & H^2 & H^3 & H^4 & \dots &  H^{4n-3} &  H^{4n-2} & H^{4n-1} \\
\hline
\ZM & 0 &  \ZM/(2) & 0 & \ZM/(2) & \dots & 0 &  \ZM/(2) & \ZM
\end{array}
\]

Now suppose that $G$ is of type $G_2$. One case use \eqref{eq:min} to
compute the equivariant multiplicity:
\begin{gather*}
\dim X = 6, \\
e_0X = \frac{18}{\prod_{\a \in \Phi_\textrm{lg}} (\delta + \a)}.
\end{gather*}
The calculation of the cohomology of
$X \setminus \{ 0 \} = \OC_{\mini}$ is performed in \cite[Section 3.9]{cohmin}:
\[
\begin{array}{cccccccccccc}
H^0 & H^1 & H^2 & H^3 & H^4 & H^5 &  H^6 & H^7 & H^8 & H^9 & H^{10} & H^{11} \\
\hline
\ZM & 0 & 0 & 0 & \ZM/(3) & 0 & \ZM/(2) & 0 & \ZM/(3) & 0 & 0 & \ZM
\end{array}
\]


\subsection{The quasi-minimal $cg_2$ singularity} In this example we let $G$ be a
simple algebraic group of type $G_2$. We use the notation of
\cite{MOV}: $\mathcal{G}_G$ denotes the affine
Grassmannian of $G$, $\om_1$ and $\om_2$ denote the fundamental
coweights (which we regard as points of $\mathcal{G}_G$), 
$\overline{\mathcal{G}_{\om_2}}$ denotes the Schubert variety
indexed by $\om_2$, and
\[
X = (L^{<0}G \cdot \om_1) \cap
\overline{\mathcal{G}_{\om_2}}
\]
is an affine normal slice to the orbit of $\om_1$ in
$\overline{\mathcal{G}_{\om_2}}$. Then $X$ is a
$\widetilde{T}$-variety where $\widetilde{T} = 
T \times \CM^*$ denotes the extended torus, and $\om_1$ is a
attractive $\widetilde{T}$-fixed point.

We have
\begin{gather*}
\dim X = 4, \\
e_xX = \frac{27}{(\a_0 + \a_1)(\a_0 + \a_1 + 3\a_2)(2\a_0 + 5\a_1 +
  6\a_2)(2\a_0 + 5\a_1 + 9\a_2)}.
\end{gather*}
The cohomology $H^\bullet(X \setminus \{ x \}; \ZM)$ is given by:
\[
\begin{array}{ccccccccc}
H^0 & H^1 & H^2 & H^3 & H^4 & H^5 & H^6 & H^7 \\
\hline
\ZM & 0 &  \ZM/(3) & 0 & \ZM/(3) & 0 & \ZM/(3) & \ZM
\end{array}
\]
(The calculation of the equivariant multiplicity was performed, using
Kumar's formula \cite{Kumar},  in
\cite{MOV}. The 
calculation of the cohomology $H^\bullet(X \setminus \{ x \}; \ZM)$ 
may be performed using moment graph techniques \cite{JW}).

\subsection{Other Kleinian singularities}
We conclude by discussing Kleinian singularities of types $D$ and $E$.
These examples are intended to convince the reader
that once one drops the assumption that $H^\bullet_T(U)$ is
torsion-free one cannot hope to have a connection between
$p$-smoothness and the equivariant multiplicity (as in our main
theorem).

Recall that if $\G \subset SL_2(\CM)$ is a finite subgroup then the
quotient $X = \CM^2 / \G$ is a Kleinian singularity. Also $X$ has a unique
singular point $x$ given by the image of $0 \in \CM^2$. If $\Gamma$ is cyclic,
then $X$ is a Kleinian singularity of type $A_{|\Gamma|-1}$ and our
main theorem applies (after appropriate choice of torus action). We have discussed this case
above.

However, if $\Gamma$ is not cyclic then $X$ still admits an 
attractive $\CM^*$-action. The equation of $X$ in $\CM^3$ as well as
the weights of $\CM^*$ are given as follows:
\begin{align*}
  D_n & : X^{n-1} + XY^2 + Z^2 = 0, &\text{weights: } & (2,n-2,n-1), \\
  E_6 & : X^4 + Y^3 + Z^2 = 0, &\text{weights: } & (3,4,6), \\
  E_7 & : X^3Y + Y^3 + Z^2 = 0, &\text{weights: } & (4,6,9), \\
  E_8 & : X^5 + Y^3 + Z^2 = 0, &\text{weights: } & (6,10,15).
\end{align*}
In each case the projection $(X,Y,Z) \mapsto (X,Y)$ induces a finite surjective map of degree $2$. Applying \eqref{degree d} one may calculate
\[
e_xX = \frac{1}{d \chi^2}
\]
where $\chi$ denotes the identity character of $\CM^*$, and $d = n-2$
in type $D_n$, and $d = 6, 12$ and $30$ in types $E_6$, $E_7$ and
$E_8$ respectively.

However in \cite{decperv} the first author has shown that $X$ is $p$-smooth if and
only if $p$ does not divide the index of connection of the
corresponding root system. In particular, $X$ is
$p$-smooth if and only if $p \ne 2$ in type $D_n$, $p \ne 3$ in type $E_6$, $p \ne
2$ in type $E_7$ and $X$ is $\ZM$-smooth in type $E_8$. One can
check directly that $H_T^\bullet(X \setminus \{ x \} , \ZM)$ has
torsion in all cases except $E_8$. (Which explains why these examples do not
contradict our main theorem!) 
Hence in these cases there seems to be no relation between those $p$ for
which $p$ is not $p$-smooth and the equivariant multiplicity.

\end{document}